\documentclass{amsart}

\usepackage[latin1]{inputenc}
\usepackage[T1]{fontenc}
\usepackage{enumerate}
\usepackage{mathtools}
\usepackage{amssymb}
\usepackage{mathrsfs}

\newtheorem{theorem}{Theorem}[section]
\newtheorem{lemma}[theorem]{Lemma}

\newtheorem{prop}[theorem]{Proposition}

\theoremstyle{definition}

\theoremstyle{remark}

\numberwithin{equation}{section}

\newcommand{\N}{\mathbb N}

\newcommand{\R}{\mathbb R}

\newcommand{\T}{\mathbb T}

\title[Common hypercyclic subspaces for $\mu D$ and $\mu T_a$]{Existence of common hypercyclic subspaces for the derivative operator and the translation operators}
\author[Q. Menet]{Quentin Menet}
\address{Quentin Menet, Univ. Artois, EA 2462, Laboratoire de Mathématiques de Lens (LML), F-62300 Lens, France}
\email{quentin.menet@univ-artois.fr}

\allowdisplaybreaks[3]
\begin{document}
\begin{abstract}
We show that the non-zero multiples of the derivative operator and the non-zero multiples of non-trivial translation operators on the space of entire functions share a common hypercyclic subspace, \emph{i.e.} a closed infinite-dimensional subspace in which each non-zero vector has a dense orbit for each of these operators.
\end{abstract}
\maketitle
\section{Introduction}

Let $X$ be a separable infinite-dimensional Fréchet space and $T$ a continuous and linear operator on $X$. We say that $T$ is hypercyclic if there exists a vector $x\in X$ (also called hypercyclic) such that the orbit $\text{Orb}(x,T):=\{T^nx: n\ge 0\}$ is dense in $X$, and we denote by $\text{HC}(T)$ the set of hypercyclic vectors for $T$. 

If $\text{HC}(T)$ is non-empty, it is known that $\text{HC}(T)$ is a dense $G_{\delta}$-set~\cite{Bir} and that $\text{HC}(T)\cup\{0\}$ contains a dense infinite-dimensional subspace~\cite{Bou, Her}. One can wonder if $\text{HC}(T)\cup\{0\}$ is also spaceable, i.e. contains a closed infinite-dimensional subspace. It was proved in~\cite{Mon} that this is not the case in general: there exist some hypercyclic operators $T$ for which $\text{HC}(T)\cup\{0\}$ is not spaceable. If $\text{HC}(T)\cup\{0\}$ is spaceable, we will say that $T$ possesses a hypercyclic subspace. The interested reader can refer to two books \cite{Bay, Gro} for more information about hypercyclic operators and to the book \cite{Aro2} for more information about spaceability.

In this paper, we investigate the multiples of the derivative operator $D$ on $H(\mathbb{C})$ defined by $Df=f'$ and the multiples of the translation operators $T_a$ on~$H(\mathbb{C})$ defined by $T_af(\cdot)=f(\cdot+a)$. Recall that $H(\mathbb{C})$ is the space of entire functions on the whole complex plane endowed with the topology of uniform convergence on compact subsets which is induced by the sequence of norms $(p_j)_{j\ge 1}$ given  by
\[p_j(f)=\sup_{|z|<j}|f(z)|.\]
These operators are the first examples of hypercyclic operators~\cite{Bir2, Mac}. In fact, each non-zero multiple of the derivative operator~\cite{Shk, Men1} and each non-zero multiple of a non-trivial translation operator~\cite{Pet} possesses a hypercyclic subspace. One can therefore wonder if a non-zero multiple of the derivative operator and a non-zero multiple of a non-trivial translation operator possess a common hypercyclic subspace, \emph{i.e.} a closed infinite-dimensional subspace in which each non-zero vector is hypercyclic for each of these operators. A sufficient condition for the existence of a common hypercyclic subspace for a countable family can be found in~\cite{Aro1}.

\begin{theorem}[Criterion $M_0$ for countable families~\cite{Aro1}]\label{M0finite}
Let $X$ be a separable Fréchet space with a continuous norm and let $\{T_{\lambda}\}_{\lambda\in\Lambda}$ be a family of operators where $\Lambda$ is at most countable. Let $M_0$ be a closed infinite-dimensional subspace of $X$. If for every $\lambda\in \Lambda$, there exists an increasing sequence $(n_k)$ such that the operator $T_{\lambda}$ satisfies the Hypercyclicity Criterion along $(n_k)$ and such that for every $x\in M_0$, $T^{n_k}_{\lambda}x$ tends to $0$, then  the family $\{T_{\lambda}\}_{\lambda\in\Lambda}$ possesses a common hypercyclic subspace.
\end{theorem}

The difficulty in using this criterion relies on the existence of the subspace~$M_0$. We will see in~Section~\ref{DT} how we can succeed to construct such a subspace for the operators $\mu D$ and $T_a$ in order to deduce the existence of a common hypercyclic subspace for these two operators.

Thanks to Costakis and Sambarino~\cite{Cos}, we also know that the family of the non-zero multiples of the derivative operator $\{\mu D\}_{\mu\ne 0}$ possesses a dense $G_{\delta}$-set of common hypercyclic vectors and that the family of non-trivial translation operators $\{T_a\}_{a\in \mathbb{C}\backslash\{0\}}$ possesses a dense $G_{\delta}$-set of common hypercyclic vectors. In 2010, Shkarin~\cite{Shk2} improved this last result by showing that the family $\{\mu T_a\}_{a,\mu \in\mathbb{C}\backslash\{0\}}$ possesses a dense $G_{\delta}$-set of common hypercyclic vectors. The family $\{\mu D\}_{\mu\ne 0}\cup\{\mu T_a\}_{a,\mu\in\mathbb{C}\backslash\{0\}}$ thus possesses a dense $G_{\delta}$-set of common hypercyclic vectors.
 
One can therefore wonder if the above families possess a common hypercyclic subspace. Since these families are uncountable, we cannot use the previous criterion. However, there also exist versions of criterion $M_0$ for uncountable families parametrized by a one-dimensional set~\cite{Bay0, Bes}. In particular, it is shown in~\cite{Bes} that the uncountable family $\{\mu D\}_{\mu\ne 0}$ possesses a common hypercyclic subspace. Unfortunately, we cannot apply these criteria to the family $\{\mu T_a\}_{a,\mu\in\mathbb{C}\backslash\{0\}}$ and it is thus not known if the family $\{\mu T_a\}_{a,\mu\in\mathbb{C}\backslash\{0\}}$ possesses a common hypercyclic subspace. However, we can easily adapt Theorem 4.1 in \cite{Bes} in order to obtain a sufficient condition for the existence of common hypercyclic subspaces which can be applied to the two-dimensional family  $\{\mu T_a\}_{a,\mu\in \mathbb{C}\backslash\{0\}}$. Indeed, if we look at the proof of  this theorem, we can remark that we only need the following properties in order to construct a common hypercyclic subspace:

\begin{theorem}[Criterion $M_0$ for uncountable families]\label{M0gen}
Let $X$ be a Fréchet space with a continuous norm, let $Y$ be a separable Fréchet space and let $\{T_{k,\lambda}\}_{\lambda\in \Lambda}$ be a family of sequences of operators in $L(X,Y)$. Suppose that there exist chains   $(\Lambda^j_n)_{n\ge 1}$ of $\Lambda$ $(j=0,1,2)$  satisfying:
\begin{enumerate}
\item[(i)]\  for every $k,n\ge 1$, the family
$
\{  T_{k, \lambda} \}_{\lambda\in \Lambda_n^0} \ \mbox{ is equicontinuous;}
$
\
\item[(ii)]\  for every $n\ge 1$, there exists a dense subset $X_{n,0}$ of $X$ such that for every $x\in X_{n,0}$,
\[
T_{k, \lambda}x \xrightarrow[k\to \infty]{} 0\quad\text{uniformly on $\lambda\in \Lambda_n^1$;}
\]
\
\item[(iii)]\ for every $l,n\ge 1$, every $\varepsilon>0$, every $y\in Y$, every $K_0\ge 0$, there exist $x\in X$ and $K_1\ge K_0$ such that $p_l(x)<\varepsilon$ and such that for every $\lambda\in \Lambda_n^2$, there exists $k\in [K_0,K_1]$ such that
\[q_l(T_{k,\lambda}x-y)<\varepsilon\]
\item[(iv)]\ there exists an infinite-dimensional closed subspace $M_0$ such that for any $(\lambda, x) \in \Lambda\times M_0$, 
\[T_{k,\lambda}x\xrightarrow[k\to \infty]{} 0.\]
\end{enumerate}
Then $\{ (T_{n,\lambda})_{n\ge 0} \}_{\lambda\in \Lambda}$ has a common hypercyclic subspace.
\end{theorem}

Thanks to this criterion, we will be able to show in Section~\ref{DT2} that the family $\{\mu T_a\}_{a,\mu\in\mathbb{C}\backslash\{0\}}$ and even the family $\{\mu D\}_{\mu\ne 0}\cup\{\mu T_a\}_{a,\mu\in\mathbb{C}\backslash\{0\}}$ possess a common hypercyclic subspace. This is the main result of this paper. 

\begin{theorem}\label{mainthm}
The family $\{\lambda D\}_{\lambda\in\mathbb{C}\backslash\{0\}}\cup \{\mu T_a\}_{a,\mu\in\mathbb{C}\backslash\{0\}}$ possesses a common hypercyclic subspace.
\end{theorem}

\section{Common hypercyclic subspace for $\mu D$ and $T_a$}\label{DT}

Let $\mu\ne 0$ and $a\ne 0$. We are interested in the existence of common hypercyclic subspaces for $\mu D$ and $T_a$. Since the family $\{\mu D, T_a\}$ is finite and since $\mu D$ and $T_a$ satisfy the Hypercyclicity Criterion along every increasing sequence, if follows from Theorem~\ref{M0finite} that it suffices to find increasing sequences $(n_k)$ and $(m_k)$ and to find a closed infinite-dimensional $M_0$ such that for every $f\in  H(\mathbb{C})$, we have
\[ (\mu D)^{n_k}f\xrightarrow[k\to\infty]{} 0\quad\text{and}\quad  T_a^{m_k}f\xrightarrow[k\to\infty]{} 0.\]
The existence of such increasing sequences $(n_k)$ and $(m_k)$ and such a closed infinite-dimensional subspace $M_0$ will follow from the following theorem.

\begin{theorem}[Criterion $(M_k)$ for finite families {\cite[Theorem 2.3]{Bes}}]\label{NiceMnfinite} Let  $X$ be an infinite-dimensional Fr\'{e}chet space with continuous norm, let $Y$ be a separable Fr\'echet space and let $\Lambda$ be a finite set.
Let $\{ (T_{k, \lambda})_{k\ge 1} \}_{\lambda\in\Lambda}$ be a family of sequences of operators in $L(X,Y)$.  Suppose that
\begin{enumerate}
\item\  there exists a dense subset $X_{0}$ of  $X$ such that for any $x\in X_{0}$, any $\lambda\in \Lambda$,
\[
T_{k,\lambda}x\xrightarrow[k\to\infty]{} 0;
\]
\item\ there exists a non-increasing sequence of infinite-dimensional closed subspaces $(M_k)$ of $X$ such that  for every $\lambda\in \Lambda$, the sequence $(T_{k,\lambda})_{k\ge 1}$ is equicontinuous along $(M_k)$.
\end{enumerate}
Then for any map $\phi:\mathbb{N}\to \mathbb{N}$ there exist an increasing sequence of integers $(k_s)_{s\ge 1}$ and an infinite-dimensional closed subspace $M_0$ of $X$ such that for any $x\in M_0$, any $\lambda\in \Lambda$,
\[T_{k,\lambda}x\xrightarrow[k\to \infty]{k\in I} 0,\]
where $I=\bigcup_{s\ge 1}[k_s,k_s+\phi(k_s)]$. 
\end{theorem}

The proof of the existence of a common hypercyclic subspace for $\mu D$ and $T_a$ can then be divided into two steps:

\begin{enumerate}
\item There exists a dense set $X_0$ and two increasing sequences of positive integers $(n_k)$ and $(m_k)$ such that for every $f\in X_0$, $(\mu D)^{n_k}f\to 0$ and $T_a^{m_k}f\to 0$.
\item There exists a non-increasing sequence of closed infinite-dimensional subspaces $(M_k)$ such that for every $j\ge 1$, there exists a continuous norm $q_j$ on $H(\mathbb{C})$ such that  for every $k\ge j$, every $f\in M_k$, we have
\[p_j((\mu D)^{n_k}f)\le q_j(f)\quad\text{and}\quad p_j(T_a^{m_k}f)\le q_j(f).\]
\end{enumerate}

In order to construct the set $X_0$ and the sequence $(M_k)$, we will repeatedly use the well-known Mergelyan's Theorem.

\begin{theorem}[Mergelyan's Theorem]
Let $K$ be a compact subset of $\mathbb{C}$ with connected complement. If $h:K\to \mathbb{C}$ is a function continuous on $K$ and holomorphic in the interior of $K$ then for every $\varepsilon>0$, there exists a polynomial $P$ such that 
\[\sup_{z\in K}\|h(z)-P(z)\|<\varepsilon.\]
\end{theorem}

We notice that if $K$ is a finite union of pairwise disjoint closed disks in  $\mathbb{C}$ then $K$ is a compact subset with connected complement. In particular, we directly deduce the following lemma from Mergelyan's Theorem.

\begin{lemma}\label{lem1}
For every $j\ge 1$, there exists $N\ge 0$ such that for every $n\ge N$, every $\varepsilon>0$, every $f\in H(\mathbb{C})$, there exists a polynomial $P$ such that
\[p_j(P)<\varepsilon\quad\text{and}\quad p_j(T^n_a(P+f))<\varepsilon.\]
\end{lemma}
\begin{proof}
Let $j\ge 1$. We consider $N> \frac{2j}{|a|}$ so that for every $n\ge N$, we have ${D(0,j)\cap D(-an,j)=\emptyset}$ where $D(x,r)=\{z\in\mathbb{C}:|z-x|\le r\}$. We then obtain the desired result by applying Mergelyan's theorem to the function $h$ defined from $D(0,j)\cup D(-an,j)$ to $\mathbb{C}$ by $h(z)=-f(z)$ if $z\in D(-an,j)$ and $h(z)=0$ otherwise.
\end{proof}

If $X_0$ is the set of polynomials then it is obvious that $(\mu D)^n P\to 0$ for every $P\in X_0$. However, if $P$ is a non-zero polynomial, we never have $T_a^{n} P\to 0$, even if we look along a subsequence. Thanks to the above lemma, we can however perturb a dense sequence of polynomials in order to get the desired convergences.

\begin{prop}\label{X0}
There exist a dense set $X_0$ and two increasing sequences of positive integers $(n_k)$ and $(m_k)$ such that for every $f\in X_0$, $(\mu D)^{n_k}f\to 0$ and $T_a^{m_k}f\to 0$.
\end{prop}
\begin{proof}
We consider a dense sequence of polynomials $(P_{l,l})_{l\ge 0}$ and we construct by induction a family of polynomials $(P_{l,k})_{l<k}$ such that the set $X_0=\{\sum_{i=l}^{\infty}P_{l,i}:l\ge 0\}$ is well-defined and satisfies the required conditions.\\
We start by letting $n_1=\text{deg}(P_{0,0})+1$ and by using Lemma~\ref{lem1} in order to obtain a polynomial $P_{0,1}$ and a positive integer $m_1$ satisfying 
\begin{itemize}
\item $p_1(P_{0,1})<\frac{1}{2}$;
\item $p_1((\mu D)^{n_1}P_{0,1})<\frac{1}{2}$;
\item $p_1(T_a^{m_1}(P_{0,0}+P_{0,1})<\frac{1}{2}$.
\end{itemize}
The second condition can be satisfied because the operator $\mu D$ is continuous and there thus exists $j\ge 1$ and $\varepsilon>0$ such that if $p_j(P_{0,1})<\varepsilon$ then $p_1((\mu D)^{n_1}P_{0,1})<\frac{1}{2}$.\\
More generally, if we assume that the family $(P_{l,j})_{l<j}$ has been constructed for every $j<k$ then we let $n_k=\max\{\text{deg}(P_{l,k-1}):l\le k-1\}+1$ and we use  Lemma~\ref{lem1} several times in order to obtain  polynomials $(P_{l,k})_{l<k}$ and a positive integer $m_k>m_{k-1}$ such that for every $l<k$
\begin{itemize}
\item $p_k(P_{l,k})<\frac{1}{2^k}$;
\item $p_k((\mu D)^{n_j}P_{l,k})<\frac{1}{2^k}$ for every $j\le k$;
\item $p_k(T_a^{m_j}P_{l,k})<\frac{1}{2^k}$ for every $j<k$;
\item $p_k\Big(T_a^{m_k}\Big(\sum_{j=l}^kP_{l,j}\Big)\Big)<\frac{1}{2^k}$.
\end{itemize}
The second and the third conditions can be satisfied by using as previously the continuity of $\mu D$ and the continuity of $T_a$.\\
Let $X_0:=\{\sum_{i=l}^{\infty}P_{l,i}:l\ge 0\}$. We first remark that each series $\sum_{i=l}^{\infty}P_{l,i}$ is convergent and that $X_0$ is dense since for every $i>l$, $p_i(P_{l,i})<\frac{1}{2^i}$. Moreover, by definition of the sequence $(n_j)$ we have $D^{n_j}P_{l,i}=0$ for every $i<j$. Therefore, if $j\ge \max\{l,k\}$, we have
\begin{align*}
p_k\Big((\mu D)^{n_j}\Big(\sum_{i=l}^{\infty}P_{l,i}\Big)\Big)&= p_k\Big((\mu D)^{n_j}\Big(\sum_{i=j}^{\infty}P_{l,i}\Big)\Big)\\
&\le \sum_{i=j}^{\infty}p_i((\mu D)^{n_j}P_{l,i})\le \sum_{i=j}^{\infty}\frac{1}{2^i}\xrightarrow[j\to +\infty]{} 0
\end{align*}
and
\begin{align*}
p_k\Big(T_a^{m_j}\Big(\sum_{i=l}^{\infty}P_{l,i}\Big)\Big)&\le 
p_j\Big(T_a^{m_j}\Big(\sum_{i=l}^{j}P_{l,i}\Big)\Big)+ \sum_{i=j+1}^{\infty} p_i(T_a^{m_j}P_{l,i})\\
&<\frac{1}{2^j}+\sum_{i=j+1}^{\infty}\frac{1}{2^i}\xrightarrow[j\to +\infty]{} 0.
\end{align*}
\end{proof}

Let $p'_j(f)=\sum_{k=0}^{\infty}|x_k|j^k$ where $f=\sum_{k=0}^{\infty}x_kz^k$. The sequence $(p'_j)$ is another increasing sequence of norms inducing the topology of $H(\mathbb{C})$ which is particularly useful when we want to construct a hypercyclic subspace for the operators $P(D)$ where $P$ is a non-constant polynomial (see~\cite{Men1}). For instance, the existence of an increasing sequence $(n_k)$ and a non-increasing sequence of closed infinite-dimensional subspaces $(M_k)$ such that $((\mu D)^{n_k})$ is equicontinuous along $(M_k)$ relies on the fact that for every $k$
\[\limsup_n\max_{m,j\le k}\frac{p'_{j}((\mu D)^{m}z^n)}{p'_{2j}(z^n)}=\limsup_n\max_{m,j\le k}\frac{|\mu|^m (\prod_{l=n-m+1}^{n}l)j^{n-m}}{(2j)^n}=0.\]
Indeed, thanks to these equalities, we can find an increasing sequence $(n_k)$ such that for every $i,j\le k$, \[p'_{j}((\mu D)^{n_i}z^{n_k})\le p'_{2j}(z^{n_k})\] and it is not difficult to prove that $((\mu D)^{n_k})$ is then equicontinuous along $(M_k)$ if we let $M_k=\overline{\text{span}}\{z^{n_l}:l\ge k\}$. 

In fact, the construction of a convenient sequence $(M_k)$ for $\mu D$ mainly relies on the valuation of considered elements. The following lemma will allow us to adapt the above reasoning in order to obtain a non-increasing sequence of closed infinite-dimensional subspaces $(M_k)$ working for both $\mu D$ and $T_a$.

\begin{lemma}\label{lem2}
Let $(m_k)$ be an increasing sequence of positive integers such that $m_1|a|>2$ and $D(-m_ja,j)\cap D(-m_ka,k)= \emptyset$ for every $j\ne k$. Then for every $k\ge 1$, every $\varepsilon>0$, every $d\ge 0$, there exists a polynomial $P$ such that
\[\text{\emph{val}}(P)\ge d,\quad p_1(P)\ge \frac{1}{2}\quad\text{and}\quad p_j(T^{m_j}_aP)<\varepsilon\quad\text{for every $j\le k$}.\]
\end{lemma}
\begin{proof}
We define a function $g$ from $D(0,1)\cup \bigcup_{j\le k}D(-am_j,j)$ to $\mathbb{C}$ by $g(z)=1$ if $z\in D(0,1)$ and $g(z)=0$ otherwise. By assumption, $g$ is well-defined and since $g$ is continuous on $D(0,1)\cup \bigcup_{j\le k}D(-am_j,j)$ and holomorphic on the interior of $D(0,1)\cup \bigcup_{j\le k}D(-am_j,j)$, we deduce from the Mergelyan's Theorem that there exists a polynomial $Q$ such that
\[p_1(Q-1)<\frac{1}{2}\quad\text{and}\quad \sup_{z\in D(-am_j,j)}|Q(z)|<\frac{\varepsilon}{(|a|m_j+j)^d}\quad\text{for every $j\le k$}.\]
Let $j\le k$. We deduce that if we let $P(z)=z^dQ(z)$, then we have $\text{val}(P)\ge d$,
\[p_j(T^{m_j}_a P)=\sup_{z\in D(-am_j,j)}|P(z)|\le (|a|m_j+j)^d\sup_{z\in D(-am_j,j)}|Q(z)|<\varepsilon\]
and since $p_1(Q-1)<\frac{1}{2}$, we have 
\[p_1(P)=\sup_{|z|=1}|z^dQ(z)|=\sup_{|z|=1}|Q(z)|\ge \frac{1}{2}.\]
\end{proof}

\begin{prop}\label{propMk}
Let $(n_k)$ and $(m_k)$ be two increasing sequences of positive integers. If $m_1|a|>2$ and $D(-m_ja,j)\cap D(-m_ka,k)= \emptyset$ for every $j\ne k$ then there exists a non-increasing sequence of closed infinite-dimensional subspaces $(M_k)$ such that for every $j\ge 1$, there exists a continuous norm $q_j$ on $H(\mathbb{C})$ such that for every $k\ge j$, every $f\in M_k$, we have
\[p_j((\mu D)^{n_k}f)\le q_j(f)\quad\text{and}\quad p_j(T_a^{m_k}f)\le q_j(f).\]
\end{prop}
\begin{proof}
Since the sequence $(p'_j)$ is an increasing sequence of norms inducing the topology of $H(\mathbb{C})$, we deduce that for every $j\ge 0$, there exists $C_j> 0$ and $l_j\ge 1$ such that for every $f\in H(\mathbb{C})$, $p_j(f)\le C_j p'_{l_j}(f)$.\\
Let $P_0=1$. We then consider a sequence $(P_k)_{k\ge 1}$ such that for every $k\ge 1$
\begin{itemize}
\item $\text{val}(P_{k})>\text{deg}(P_{k-1})$;
\item $p_1(P_k)\ge \frac{1}{2}$;
\item $\text{val}(P_k)\ge N_k$ where $N_k$ satisfies
\[\sup_{n\ge N_k}\max_{i,j\le k}\frac{p'_{l_j}((\mu D)^{n_i}e_n)}{p'_{2l_j}(e_n)}\le 1.\]
\item $p_j(T_a^{m_j}P_k)\le \frac{1}{2^k}$ for every $j\le k$.
\end{itemize}
The existence of $N_k$ follows from the fact that
\begin{align*}
\limsup_n\max_{i,j\le k}\frac{p'_{l_j}((\mu D)^{n_i}e_n)}{p'_{2l_j}(e_n)}
&=\limsup_n\max_{i,j\le k}\frac{l_j^{n-n_i}|\mu|^{n_i}\prod_{\nu=n-n_i+1}^{n}\nu}{(2l_j)^n}\\
&\le \limsup_n\frac{\max\{1,|\mu|\}^{n_k}n^{n_k}}{2^n}=0.
\end{align*}
and the existence of the sequence $(P_k)_{k\ge 1}$ follows from Lemma~\ref{lem2}.

Since $\text{val}(P_{k+1})>\text{deg}(P_k)$, the sequence $(P_k)_{k\ge 1}$ is a basic sequence in $H(\mathbb{C})$. We let $M_k=\overline{\text{span}}\{P_l:l \ge k\}$ and we show that this sequence of closed infinite-dimensional subspaces satisfies the required conditions. We first remark that if $f=\sum_{l=k}^{\infty}a_lP_l$ is convergent then the sequence $(a_l)_{l\ge k}$ is bounded by $2C_1 p'_{l_1}(f)$. Indeed, since $p_1(P_l)\ge \frac{1}{2}$, we have
\[|a_l|\le 2p_1(a_lP_l)\le 2C_1p'_{l_1}(a_lP_l)\le 2C_1 p'_{l_1}(f).\]
Let $1\le j\le k$ and $f=\sum_{l=k}^{\infty}a_lP_l\in M_k$. We have thus
\[p_j(T_a^{m_k}f)\le \sum_{l=k}^{\infty} p_k(T_a^{m_k}(a_lP_l))\le \sum_{l=k}^{\infty}\frac{|a_l|}{2^l}\le 2C_1 p'_{l_1}(f)\]
and
\begin{align*}
p_j((\mu D)^{n_k}f)\le C_jp'_{l_j}((\mu D)^{n_k}f)
&= \sum_{l=k}^{\infty} C_jp'_{l_j}((\mu D)^{n_k}(a_lP_l))\\
&\le \sum_{l=k}^{\infty} C_jp'_{2l_j}(a_lP_l)\quad\text{since val}(P_l)\ge N_l\\
&=C_jp'_{2l_j}(f).
\end{align*}
We have thus the desired result if we consider $q_j=\max\{2C_1 p'_{l_1},C_jp'_{2l_j}\}$.
\end{proof}

\begin{theorem}
The operators $\mu D$ and $T_a$ possess a common hypercyclic subspace.
\end{theorem}
\begin{proof}
By Proposition~\ref{X0}, there exist a dense set $X_0$ and two increasing sequences of positive integers $(n_k)$ and $(m_k)$ such that for every $f\in X_0$, $(\mu D)^{n_k}f\to 0$ and $T_a^{m_k}f\to 0$. Without loss of generality, we can assume that $m_1|a|>2$ and $D(-m_ja,j)\cap D(-m_ka,k)\ne \emptyset$ for every $j\ne k$. Thanks to Proposition~\ref{propMk}, we then get the existence of a non-increasing sequence of infinite-dimensional closed subspaces $(M_k)$ of $X$ such that $((\mu D)^{n_k})_k$ and $(T^{m_k}_a)_k$ are equicontinuous along $(M_k)$. We now deduce from Theorem~\ref{NiceMnfinite} applied with $T_{k,1}=(\mu D)^{n_k}$ and $T_{k,2}=T^{m_k}_a$ that there exist a closed infinite-dimensional closed subspace $M_0$ and subsequences $(n'_k)$ of $(n_k)$ and $(m'_k)$ of $(m_k)$ such that for every $f\in M_0$,
\[(\mu D)^{n'_k}f\to 0 \quad\text{and}\quad T^{m'_k}_af\to 0.\]
Since $\mu D$ and $T_a$ satisfy the Hypercyclicity Criterion along each increasing sequence and thus in particular along $(n'_k)$ and $(m'_k)$, the desired result follows from Theorem~\ref{M0finite}.
\end{proof}

\section{Common hypercyclic subspace for $\{\mu D\}_{\mu\in\mathbb{C}\backslash\{0\}}$ and $\{\mu T_a\}_{a,\mu\in \mathbb{C}\backslash\{0\}}$}\label{DT2}

In this section, we are interested in the family $\{\mu D\}_{\mu \in\mathbb{C}\backslash\{0\}}\cup\{\mu T_a\}_{a,\mu\in \mathbb{C}\backslash\{0\}}$. We remark that the considered family is now infinite and even uncountable. Fortunately, in order to show that the families $\{\mu D\}_{\mu \in\mathbb{C}\backslash\{0\}}$ and $\{\mu T_a\}_{a,\mu\in \mathbb{C}\backslash\{0\}}$ possess a common hypercyclic subspace, it suffices to show that the families $\{e^b D\}_{b\in \R}$ and $\{e^b T_a\}_{b\in \R, a\in \mathbb{T}}$ possess a common hypercyclic subspace since $\text{HC}(\lambda T)=\text{HC}(T)$ if $\lambda\in \T$ (see~\cite{Leo}) and since $\text{HC}(T_a)=\text{HC}(T_b)$ if $\frac{a}{b}\in\mathbb{R}^+$( see~\cite{Con}).

The existence of common hypercyclic subspaces for the families $\{e^b D\}_{b\in \R}$ and $\{e^b T_a\}_{b\in \R, a\in \mathbb{T}}$ will be obtained by applying Theorem~\ref{M0gen} to the family $\{T_{k,\lambda}\}_{\lambda\in \Lambda}$ given by
\begin{itemize}
\item $\Lambda:=\mathbb{R}\cup(\T\times\mathbb{R})$, 
\item $T_{k,b}=e^b D^{\tilde{n}_k}$ for every $b\in \mathbb{R}$,
\item  $T_{k,(a,b)}=e^bT^{\tilde{m}_k}_a$ for every $a\in\T\backslash\{1\}$, every $b\in\R$, 
\item $T_{k,(1,b)}=e^b T^{\tilde{t}_k}_1$ for every $b\in\R$
\end{itemize}
and by considering the chains $(\Lambda_n)$ given by 
\[\Lambda_n=[-n,n]\cup \big(\T\times [-n,n]\big)\quad \text{or} \quad[-n,n]\cup\big((\mathbb{T}\backslash\{e^{i\theta}:|\theta|<\frac{1}{n}\})\times [-n,n]\big)\cup\big(\{1\}\times [-n,n]\big).\]
The existence of a suitable subspace $M_0$ will be again obtained thanks to Criterion~$(M_k)$.

\begin{theorem}[Criterion $(M_k)$ for infinite families {\cite[Theorem 2.3]{Bes}}]\label{NiceMn}
Let  $X$ be an infinite-dimensional Fr\'{e}chet space with continuous norm, let $Y$ be a separable Fr\'echet space and let $\Lambda$ be a set.
Let $\{ (T_{k, \lambda})_{k\ge 1} \}_{\lambda\in\Lambda}$ be a family of sequences of operators in $L(X,Y)$.  Suppose that there exist chains   $(\Lambda^j_n)_{n\ge 1}$ of $\Lambda$ $(j=0,1,2)$  satisfying:
\begin{enumerate}
\item[(i)]\  for each $n\in\mathbb{N}$ and each $k\in \N$, the family
$
\{  T_{k, \lambda} \}_{\lambda\in \Lambda_n^0} \ \mbox{ is equicontinuous;}
$
\item[(ii)]\  for each $n\in\mathbb{N}$, there exists a dense subset $X_{n,0}$ of  $X$ such that for any $x\in X_{n,0}$,
\[
T_{k,\lambda}x\xrightarrow[k\to\infty]{} 0 \ \ \ \mbox{uniformly on $\lambda\in \Lambda_n^1$;}
\]
\item[(iii)]\ there exists a non-increasing sequence of infinite-dimensional closed subspaces $(M_k)$ of $X$ such that for each $n\ge 1$,  the family of sequences
\[
\{   (T_{k,\lambda})_{k\ge 1}\}_{\lambda\in \Lambda_n^2}
\]
is uniformly equicontinuous along $(M_k)$.
\end{enumerate}
Then for any map $\phi:\mathbb{N}\to \mathbb{N}$ there exist an increasing sequence of integers $(k_s)_{s\ge 1}$ and an infinite-dimensional closed subspace $M_0$ of $X$ such that for any $(x,\lambda)\in M_0\times \Lambda$,
\[T_{k,\lambda}x\xrightarrow[k\to \infty]{k\in I} 0,\]
where $I=\bigcup_{s\ge 1}[k_s,k_s+\phi(k_s)]$. 
\end{theorem}

As previously, we will use repeatedly Mergelyan's Theorem in order to construct the set $X_0$ and the sequence $(M_k)$. In view of our operators, we would like to consider the compact sets $D(0,j)\cup\Big[\bigcup_{k=k_0}^{k_1}\bigcup_{a\in \mathbb{T}}D(-ak,j)\Big]$. Unfortunately, these sets have not a connected complement. For this reason, we will work with the following compact sets \[D(0,j)\cup\Big[\bigcup_{k=k_0}^{k_1}\bigcup_{a\in \mathbb{T}\backslash \{e^{i\alpha}:\alpha\in ]-\theta,\theta[\}}D(-ak,j)\Big]\cup \bigcup_{k=k_2}^{k_3}D(-k,j)\] which have a connected complement if $\theta>0$ and $k_0$ and $k_2-k_1$ are sufficiently big
.

\begin{lemma}\label{lem1genbis}
Let $\phi:\mathbb{N}\to \mathbb{N}$ be a map and let $(m_k)$ and $(t_k)$ be two increasing sequences. For every $\theta\in ]0,\pi[$, every $j\ge 1$, every $K\ge 0$, there exist $k_1,k_2\ge K$ such that for every $\varepsilon>0$, every $b>0$, every $f\in H(\mathbb{C})$, there exists a polynomial $P$ such that
\begin{enumerate}
\item $p_j(P)<\varepsilon$;
\item for every $a\in \mathbb{T}\backslash \{e^{i\alpha}:\alpha\in ]-\theta,\theta[\}$, every $k\in [m_{k_1},m_{k_1}+\phi(m_{k_1})]$,
\[p_j(e^{bk}T_a^{k}(P+f))<\varepsilon;\]
\item for every $k\in [t_{k_2},t_{k_2}+\phi(t_{k_2})]$,
\[p_j(e^{bk}T_1^{k}(P+f))<\varepsilon.\]
\end{enumerate}
\end{lemma}
\begin{proof}
Let $\Omega_k:=\bigcup_{a\in \mathbb{T}\backslash \{e^{i\alpha}:\alpha\in ]-\theta,\theta[\}}D(-ak,j)$. There exists $k_1\ge K$ such that $D(0,j)\cap \bigcup_{k=m_{k_1}}^{m_{k_1}+\phi(m_{k_1})}\Omega_k=\emptyset$ and such that the complement of $D(0,j)\cup\bigcup_{k=m_{k_1}}^{m_{k_1}+\phi(m_{k_1})}\Omega_k$ is connected. We can also find $k_2\ge K$ such that $\bigcup_{k=t_{k_2}}^{t_{k_2}+\phi(t_{k_2})}D(-k,j)$ is disjoint of $\bigcup_{k=m_{k_1}}^{m_{k_1}+\phi(m_{k_1})}\Omega_k$ and of $D(0,j)$, and such that the complement of $\Omega:=D(0,j)\cup\bigcup_{k=m_{k_1}}^{m_{k_1}+\phi(m_{k_1})}\Omega_k \cup \bigcup_{k=t_{k_2}}^{t_{k_2}+\phi(t_{k_2})}D(-k,j)$ is connected.\\
We define a function $h$ from $\Omega$ to $\mathbb{C}$ by $h(z)=0$ if $z\in D(0,j)$ and $h(z)=-f(z)$ otherwise. Since $h$ is continuous on $\Omega$ and holomorphic on the interior of $\Omega$, we deduce from the Mergelyan's Theorem that there exists a polynomial $P$ such that
\[p_j(P)<\varepsilon\quad\text{and}\quad \sup_{z\in \Omega\backslash D(0,j)}\big|P(z)+f(z)\big|<\frac{\varepsilon}{e^{bC}},\]
where $C=\max\{m_{k_1}+\phi(m_{k_1}),t_{k_2}+\phi(t_{k_2})\}$. We deduce that
\begin{enumerate}
\item for every $a\in \mathbb{T}\backslash \{e^{i\alpha}:\alpha\in ]-\theta,\theta[\}$, every $k\in [m_{k_1},m_{k_1}+\phi(m_{k_1})]$,
\[p_j(e^{bk}T^{k}_a(P+f))\le e^{bC}\sup_{z\in \Omega_k}\big|P(z)+f(z)\big|<\varepsilon;\]
\item for every  $k\in [t_{k_2},t_{k_2}+\phi(t_{k_2})]$,
\[p_j(e^{bk}T^{k}_1(P+f))\le e^{bC}\sup_{z\in D(-k,j)}\big|P(z)+f(z)\big|<\varepsilon;\]
\end{enumerate}
\end{proof}

Given a map $\phi:\mathbb{N}\to \mathbb{N}$, we say that a sequence $(n_k)_{k\ge 1}$ is $\phi$-increasing if for every $k\ge 1$, we have
\[n_{k+1}>n_{k}+\phi(n_k).\]
Thanks to Lemma~\ref{lem1genbis}, we can then prove the following proposition which will allow us to show that the condition (ii) of Theorem~\ref{NiceMn} is satisfied by our family $(T_{k,\lambda})$.

\begin{prop}\label{propX0gen}
Let $\phi:\mathbb{N}\to \mathbb{N}$ and let $(n_k)$, $(m_k)$ and $(t_k)$ be three $\phi$-increasing sequences. There exist a dense subset $X_0$ of $H(\mathbb{C})$ and three increasing sequences $(k_{0,s})$, $(k_{1,s})$ and $(k_{2,s})$ such that for every $b>0$, every $\theta\in ]0,\pi[$, every $x\in X_{0}$,
\begin{enumerate}
\item $e^{bk}D^{k}x \xrightarrow[k\to \infty]{k\in \mathcal{N}} 0$;\\
\item $e^{bk}T^{k}_{a}x\xrightarrow[k\to \infty]{k\in \mathcal{M}} 0$ uniformly on $a\in \mathbb{T}\backslash \{e^{i\alpha}:\alpha\in ]-\theta,\theta[\}$;\\
\item $e^{bk}T^{k}_{1}x\xrightarrow[k\to \infty]{k\in \mathcal{T}} 0$;
\end{enumerate}
where $\mathcal{N}=\bigcup_{s\ge 1}[n_{k_{0,s}},n_{k_{0,s}}+\phi(n_{k_{0,s}})]$,
$\mathcal{M}=\bigcup_{s\ge 1}[m_{k_{1,s}},m_{k_{1,s}}+\phi(m_{k_{1,s}})]$
and $\mathcal{T}=\bigcup_{s\ge 1}[t_{k_{2,s}},t_{k_{2,s}}+\phi(t_{k_{2,s}})]$.
\end{prop}
\begin{proof}
Let $\phi$ be a map from $\mathbb{N}$ to $\mathbb{N}$ and let $(n_k)$, $(m_k)$ and $(t_k)$ be three $\phi$-increasing sequences. We consider a dense sequence of polynomials $(P_{l,l})_{l\ge 0}$ and we construct by induction three increasing sequences $(k_{0,s})$, $(k_{1,s})$ and $(k_{2,s})$ and a family of polynomials $(P_{l,k})_{l<k}$ such that the set $X_0=\{\sum_{i=l}^{\infty}P_{l,i}:l\ge 0\}$ is well-defined and satisfies the required conditions.\\
We start by choosing $n_{k_{0,1}}>\text{deg}(P_{0,0})$ and by using Lemma~\ref{lem1genbis} in order to obtain a polynomial $P_{0,1}$ and two positive integers $k_{1,1}$ and $k_{2,1}$ satisfying 
\begin{itemize}
\item $p_1(P_{0,1})<\frac{1}{2}$;
\item $p_1(e^{k}D^{k}P_{0,1})<\frac{1}{2}$ for every $k\in [n_{k_{0,1}},n_{k_{0,1}}+\phi(n_{k_{0,1}})]$;
\item $p_1(e^{k}T_a^{k}(P_{0,0}+P_{0,1}))<\frac{1}{2}$ for every $k\in [m_{k_{1,1}},m_{k_{1,1}}+\phi(m_{k_{1,1}})]$ and every $a\in \mathbb{T}\backslash \{e^{i\alpha}:\alpha\in ]-1,1[\}$;
\item $p_1(e^{k}T_1^{k}(P_{0,0}+P_{0,1}))<\frac{1}{2}$ for every $k\in [t_{k_{2,1}},t_{k_{2,1}}+\phi(t_{k_{2,1}})]$;
\end{itemize}
More generally, if we assume that the family $(P_{l,j})_{l<j}$ has been constructed for every $j<s$ then we choose $n_{k_{0,s}}>\max\{\text{deg}(P_{l,s-1}):l\le s-1\}$ with $k_{0,s}>k_{0,s-1}$ and we use  Lemma~\ref{lem1genbis} several times in order to obtain polynomials $(P_{l,s})_{l<s}$ and two positive integers $k_{1,s}>k_{1,s-1}$ and $k_{2,s}>k_{2,s-1}$ such that for every $l<s$
\begin{itemize}
\item $p_s(P_{l,s})<\frac{1}{2^s}$;
\item $p_s(e^{sk}D^{k}P_{l,s})<\frac{1}{2^s}$ for every $k\in \bigcup_{j\le s}[n_{k_{0,j}},n_{k_{0,j}}+\phi(n_{k_{0,j}})]$;
\item $p_s(e^{sk}T_a^{k}P_{l,s})<\frac{1}{2^s}$ for every $k\in \bigcup_{j< s}[m_{k_{1,j}},m_{k_{1,j}}+\phi(m_{k_{1,j}})]$ and every $a\in \mathbb{T}$;
\item $p_s(e^{sk}T_1^{k}P_{l,s})<\frac{1}{2^s}$ for every $k\in \bigcup_{j< s}[t_{k_{2,j}},t_{k_{2,j}}+\phi(t_{k_{2,j}})]$;
\item $p_s\Big(e^{sk}T_a^{k}\Big(\sum_{j=l}^sP_{l,j}\Big)\Big)<\frac{1}{2^s}$ for every $k\in [m_{k_{1,s}},m_{k_{1,s}}+\phi(m_{k_{1,s}})]$ and every $a\in \mathbb{T}\backslash \{e^{i\alpha}:\alpha\in ]-\frac{1}{s},\frac{1}{s}[\}$;
\item $p_s\Big(e^{sk}T_1^{k}\Big(\sum_{j=l}^sP_{l,j}\Big)\Big)<\frac{1}{2^s}$ for every $k\in [t_{k_{2,s}},t_{k_{2,s}}+\phi(t_{k_{2,s}})]$
\end{itemize}
Let $X_0:=\{\sum_{i=l}^{\infty}P_{l,i}:l\ge 0\}$. We first remark that each series $\sum_{i=l}^{\infty}P_{l,i}$ is convergent and that $X_0$ is dense since for every $l<i$, $p_i(P_{l,i})<\frac{1}{2^i}$. We can now prove the desired properties:
\begin{enumerate}
\item for every $l,s\ge 1$, every $b>0$, every $j\ge \max\{l,s,b\}$ and every $k\in [n_{k_{0,j}},n_{k_{0,j}}+\phi(n_{k_{0,j}})]$,
\begin{align*}
p_s\Big(e^{bk}D^{k}\Big(\sum_{i=l}^{\infty}P_{l,i}\Big)\Big)&= p_s\Big(e^{bk}D^{k}\Big(\sum_{i=j}^{\infty}P_{l,i}\Big)\Big)\\
&\le \sum_{i=j}^{\infty}p_i(e^{ik}D^{k}P_{l,i})\le \sum_{i=j}^{\infty}\frac{1}{2^i}\to 0;
\end{align*}
\item for every $l,s\ge 1$, every $b>0$, every $\theta\in ]0,\pi[$, every $j\ge \max\{l,s,b,\frac{1}{\theta}\}$, every $k\in [m_{k_{1,j}},m_{k_{1,j}}+\phi(m_{k_{1,j}})]$, every $a\in \mathbb{T}\backslash \{e^{i\alpha}:\alpha\in ]-\theta,\theta[\}$, we have 
\begin{align*}
p_s\Big(e^{bk}T_a^{k}\Big(\sum_{i=l}^{\infty}P_{l,i}\Big)\Big)&\le 
p_j\Big(e^{jk}T_a^{k}\Big(\sum_{i=l}^{j}P_{l,i}\Big)\Big)+ \sum_{i=j+1}^{\infty} p_i(e^{ik}T_a^{k}P_{l,i})\\
&<\frac{1}{2^j}+\sum_{i=j+1}^{\infty}\frac{1}{2^i}\to 0.
\end{align*}
\item for every $l,s\ge 1$, every $b>0$, every $j\ge \max\{l,s,b\}$ and every $k\in [t_{k_{2,j}},t_{k_{2,j}}+\phi(t_{k_{2,j}})]$,
\begin{align*}
p_s\Big(e^{bk}T_1^{k}\Big(\sum_{i=l}^{\infty}P_{l,i}\Big)\Big)&\le 
p_j\Big(e^{jk}T_1^{k}\Big(\sum_{i=l}^{j}P_{l,i}\Big)\Big)+ \sum_{i=j+1}^{\infty} p_i(e^{ik}T_1^{k}P_{l,i})\\
&<\frac{1}{2^j}+\sum_{i=j+1}^{\infty}\frac{1}{2^i}\to 0.
\end{align*}
\end{enumerate}
\end{proof}

We now need to show that the condition (iii) of Theorem~\ref{NiceMn} is also satisfied by our family $(T_{k,\lambda})$. In order to construct the required subspaces $M_k$, a control on the valuation of the polynomials given by Mergelyan's Theorem is necessary. We thus prove the following result.

\begin{lemma}\label{lem2gen} Let $\phi:\mathbb{N}\to \mathbb{N}$ be a map and let $(m_k)$ and $(t_k)$ two $\phi$-increasing sequences. There exist two increasing sequences $(k_{1,s})$ and $(k_{2,s})$ such that for every $b>0$, every $\varepsilon>0$, every $d\ge 0$, every $s\ge 1$, there exists a polynomial $P$ such that
\begin{enumerate}
\item $\text{val}(P)\ge d$ and $p_1(P)\ge \frac{1}{2}$;
\item for every $j\le s$, every $a\in \mathbb{T}\backslash \{e^{i\alpha}:\alpha\in ]-\frac{1}{j},\frac{1}{j}[\}$, every $k\in [m_{k_{1,j}},m_{k_{1,j}}+\phi(m_{k_{1,j}})]$,
\[p_j(e^{bk}T^{k}_aP)<\varepsilon;\]
\item for every $j\le s$, every $k\in [t_{k_{2,j}},t_{k_{2,j}}+\phi(t_{k_{2,j}})]$,
\[p_j(e^{bk}T^{k}_1P)<\varepsilon.\]
\end{enumerate}
\end{lemma}
\begin{proof}
Let $\Omega_{k,s}=\bigcup_{a\in \mathbb{T}\backslash \{e^{i\alpha}:\alpha\in ]-\frac{1}{s},\frac{1}{s}[\}} D(-ak,s)$. We first remark that we can construct by induction two increasing sequences $(k_{1,s})$ and $(k_{2,s})$ such that the sets $D(0,1)$, $\bigcup_{k=m_{k_{1,s}}}^{m_{k_{1,s}}+\phi(m_{k_{1,s}})}\Omega_{k,s}$ ($s\ge 1$) and $\bigcup_{k=t_{k_{2,s}}}^{t_{k_{2,s}}+\phi(t_{k_{2,s}})}D(-k,s)$ ($s\ge 1$) are pairwise disjoint and such that for every $s\ge 1$, the complement of $\Omega_s:=D(0,1)\cup \bigcup_{j\le s}\bigcup_{k=m_{k_{1,j}}}^{m_{k_{1,j}}+\phi(m_{k_{1,j}})}\Omega_{k,j} \cup \bigcup_{j\le s}\bigcup_{k=t_{k_{2,s}}}^{t_{k_{2,s}}+\phi(t_{k_{2,s}})}D(-k,j)$ is connected.\\
Let $b>0$, $\varepsilon>0$, $d\ge 0$ and $s\ge 1$. We can define a function $h$ from $\Omega_s$ to $\mathbb{C}$ by $h(z)=1$ if $z\in D(0,1)$ and $h(z)=0$ otherwise. By assumption, $h$ is well-defined and since $h$ is continuous on $\Omega_s$ and holomorphic on the interior of $\Omega_s$, we deduce from the Mergelyan's Theorem that there exists a polynomial $Q$ such that
\[p_1(Q-1)<\frac{1}{2}\quad\text{and}\quad \sup_{z\in \Omega_s\backslash D(0,1)}|Q(z)|<\frac{\varepsilon}{e^{bC}C^d},\]
where $C=\sup\{|z|:z\in \Omega_s\}$.\\
We deduce that if we let $P(z)=z^dQ(z)$, then 
\begin{enumerate}
\item $\text{val}(P)\ge d$ and 
\[p_1(P)=\sup_{|z|=1}|z^dQ(z)|=\sup_{|z|=1}|Q(z)|\ge \frac{1}{2};\]
\item for every $j\le s$, every $a\in \mathbb{T}\backslash \{e^{i\alpha}:\alpha\in ]-\frac{1}{j},\frac{1}{j}[\}$, every $k\in [m_{k_{1,j}},m_{k_{1,j}}+\phi(m_{k_{1,j}})]$,
\[p_j(e^{bk}T^{k}_aP)=e^{bk}\sup_{z\in D(-ak,j)}|z^dQ(z)|\le e^{bC}C^d\sup_{z\in \Omega_{k,j}}|Q(z)|<\varepsilon;\]
\item for every $j\le s$, every $k\in [t_{k_{2,j}},t_{k_{2,j}}+\phi(t_{k_{2,j}})]$,
\[p_j(e^{bk}T^{k}_1P)=e^{bk}\sup_{z\in D(-k,j)}|z^dQ(z)|\le e^{bC}C^d\sup_{z\in D(-k,j)}|Q(z)|<\varepsilon.\]
\end{enumerate}
\end{proof}

\begin{prop}\label{propMkgen}
Let $\phi:\mathbb{N}\to \mathbb{N}$ and let $(n_k)$, $(m_k)$ and $(t_k)$ be $\phi$-increasing sequences. There exist increasing sequences $(k_{1,s})$ and $(k_{2,s})$, and a non-increasing sequence of infinite-dimensional closed subspaces $(M_s)$ of $H(\mathbb{C})$ such that for every $b>0$, every $\theta\in ]0,\pi[$, every $j\ge 1$, there exists a continuous seminorm $q$ of $H(\mathbb{C})$ such that for every $a\in \mathbb{T}\backslash \{e^{i\alpha}:\alpha\in ]-\theta,\theta[\}$, every $s\in\mathbb{N}$, every $x\in M_s$, we have
\begin{enumerate}
\item for every $k\in [n_{s},n_{s}+\phi(n_{s})]$, $p_j(e^{bk}D^{k}x) \le q(x)$;
\item for every $k\in [m_{k_{1,s}},m_{k_{1,s}}+\phi(m_{k_{1,s}})]$, $p_j(e^{bk}T^{k}_ax) \le q(x)$;
\item for every $k\in [t_{k_{2,s}},k_{2,s}+\phi(t_{k_{2,s}})]$, $p_j(e^{bk}T^{k}_1x) \le q(x)$.
\end{enumerate} 
\end{prop}
\begin{proof}
We recall that we let $p'_j(f)=\sum_{k=0}^{\infty}|a_k|j^k$ where $f(z)=\sum_{k=0}^{\infty}a_kz^k$ and that for every $j\ge 0$, there exists $C_j> 0$ and $l_j$ such that for every $f\in H(\mathbb{C})$, $p_j(f)\le C_j p'_{l_j}(f)$.

Let $P_0=1$ and let $(k_{1,s})$ and $(k_{2,s})$ be the sequences given by Lemma~\ref{lem2gen}. We then consider a sequence $(P_s)_{s\ge 1}$ of polynomials such that for every $s\ge 1$
\begin{itemize}
\item $\text{val}(P_{s})>\text{deg}(P_{s-1})$ and $p_1(P_s)\ge \frac{1}{2}$;
\item $\text{val}(P_s)\ge N_s$ where $N_s$ satisfies
\[\sup_{n\ge N_s}\max_{k\le n_s+\phi(n_s)}\max_{j\le s}\frac{p'_{l_j}(e^{s(n_s+\phi(n_s))}D^{k}e_n)}{p'_{l_j+1}(e_n)}\le 2.\]
\item $p_j(e^{sk}T_a^{k}P_s)\le \frac{1}{2^s}$ for every $j\le s$, every $a\in \mathbb{T}\backslash \{e^{i\alpha}:\alpha\in ]-\frac{1}{j},\frac{1}{j}[\}$ and every $k\in [m_{k_{1,j}},m_{k_{1,j}}+\phi(m_{k_{1,j}})]$.
\item $p_j(e^{sk}T^{k}_1P_s)<\frac{1}{2^s}$ for every $j\le s$, every $k\in [t_{k_{2,j}},t_{k_{2,j}}+\phi(t_{k_{2,j}})]$.
\end{itemize}
The existence of such a sequence of polynomials follows from the choice of sequences $(k_{1,s})$ and $(k_{2,s})$.\\
Since $\text{val}(P_{k+1})>\text{deg}(P_k)$, the sequence $(P_k)_{k\ge 1}$ is a basic sequence in $H(\mathbb{C})$. We let $M_k=\overline{\text{span}}\{P_l:l \ge k\}$ and we show that this sequence of closed infinite-dimensional subspaces satisfies the required conditions. We first remark that if $x=\sum_{l=k}^{\infty}a_lP_l$ is convergent then the sequence $(a_l)$ is bounded by $2C_1 p'_{l_1}(x)$ since
\[|a_l|\le 2p_1(a_lP_l)\le 2C_1p'_{l_1}(a_lP_l)\le 2C_1 p'_{l_1}(x).\]
\begin{enumerate}
\item Let $b>0$ and $j\ge 1$. 
For every $s\ge \max\{b,j\}$, every $x=\sum_{l=s}^{\infty}a_lP_l\in M_s$, we have for every $k\in [n_s,n_s+\phi(n_s)]$,
\begin{align*}
p_j(e^{b k}D^{k}x)\le C_jp'_{l_j}(e^{bk}D^{k}x)
&\le \sum_{l=s}^{\infty} C_jp'_{l_j}(e^{b(n_s+\phi(n_s))}D^{k}a_lP_l)\\
&\le \sum_{l=s}^{\infty} 2C_jp'_{l_j+1}(a_lP_l)\quad\text{since val}(P_l)\ge N_l\\
&=2C_jp'_{l_j+1}(x).
\end{align*}
and by continuity, there exists a continuous seminorm $q_{0,b,j}$ such that for every $x\in H(\mathbb{C})$, every $s< \max\{b,j\}$, every $k\in [n_s,n_s+\phi(n_s)]$
\[p_j((e^{b k}D^{k}x))\le q_{0,b,j}(x).\]
We conclude that if we let $q=\max\{q_{0,b,j},2C_jp'_{l_j+1}\}$, then we have 
\[p_j(e^{bn_s}D^{n_s}x) \le q(x)\]
for every $s\in\mathbb{N}$, every $x\in M_s$.\\
\item Let $b>0$, $\theta\in ]0,\pi[$ and $j\ge 1$. 
For every $s\ge \max\{b,\frac{1}{\theta},j\}$, every $x=\sum_{l=s}^{\infty}a_lP_l\in M_s$, every $a\in \mathbb{T}\backslash \{e^{i\alpha}:\alpha\in ]-\theta,\theta[\}$, every $k\in [m_{k_{1,s}},m_{k_{1,s}}+\phi(m_{k_{1,s}})]$, we have
\[p_j(e^{bk}T_a^{k}x)\le \sum_{l=s}^{\infty} p_j(e^{bk}T_a^{k}(a_lP_l))\le \sum_{l=s}^{\infty}\frac{|a_l|}{2^l}\le 2C_1 p'_{l_1}(x).\]
Moreover, by equicontinuity, there exists a continuous seminorm $q_{1,b,j}$ such that for every $x\in H(\mathbb{C})$, every $s< \max\{b,\frac{1}{\theta},j\}$, every $a\in \mathbb{T}$, every $k\in [m_{k_{1,s}},m_{k_{1,s}}+\phi(m_{k_{1,s}})]$,
\[p_j(e^{bk}T_a^{k}x)\le q_{1,b,j}(x).\]
We conclude that if we let $q=\max\{q_{1,b,j},2C_1 p'_{l_1}\}$, then we have 
\[p_j(e^{bk}T_a^{k}x) \le q(x)\]
for every $s\in\mathbb{N}$, every $x\in M_s$, every $k\in [m_{k_{1,s}},m_{k_{1,s}}+\phi(m_{k_{1,s}})]$ and every $a\in \mathbb{T}\backslash \{e^{i\alpha}:\alpha\in ]-\theta,\theta[\}$.\\
\item Let $b>0$ and $j\ge 1$. 
For every $s\ge \max\{b,j\}$, every $x=\sum_{l=s}^{\infty}a_lP_l\in M_s$, every $k\in [t_{k_{2,s}},t_{k_{2,s}}+\phi(t_{k_{2,s}})]$, we have
\[p_j(e^{bk}T_1^{k}x)\le \sum_{l=s}^{\infty} p_j(e^{bk}T_1^{k}(a_lP_l))\le \sum_{l=s}^{\infty}\frac{|a_l|}{2^l}\le 2C_1 p'_{l_1}(x).\]
and by continuity, there exists a continuous seminorm $q_{2,b,j}$ such that for every $x\in H(\mathbb{C})$, every $s<\max\{b,j\}$, every $k\in [t_{k_{2,s}},t_{k_{2,s}}+\phi(t_{k_{2,s}})]$,
\[p_j((e^{b k}T_1^{k}x))\le q_{2,b,j}(x).\]
We conclude that if we let $q=\max\{q_{2,b,j},2C_1 p'_{l_1}\}$, then we have 
\[p_j(e^{bk}T_1^{k}x) \le q(x)\]
for every $s\in\mathbb{N}$, every $x\in M_s$, every $k\in [t_{k_{2,s}},t_{k_{2,s}}+\phi(t_{k_{2,s}})]$.
\end{enumerate}
\end{proof}

In summary, by using Propositon~\ref{propX0gen} and Proposition~\ref{propMkgen}, we are now able to show that our family $(T_{k,\lambda})$ satisfies the conditions of Theorem~\ref{NiceMn}. In other words, we can deduce that the condition (iv) of Theorem~\ref{M0gen} is satisfied. Moreover, it follows from Propositon~\ref{propX0gen} that the condition (ii) of Theorem~\ref{M0gen} is also satisfied. Since the equicontinuity will not be difficult to obtain, it us remains to prove that the condition (iii) of Theorem~\ref{M0gen} can be satisfied. This proof will follow from the following two lemmas used by Shkarin~\cite{Shk2} to prove the existence of common hypercyclic vectors for the family $\{\mu T_a\}_{\mu,a\in \mathbb{C}\backslash\{0\}}$

\begin{lemma}[{\cite[Lemma 3.4]{Shk2}}]\label{lemS1}
For each $\delta,C>0$, there is $R>0$ such that for any $n\in \mathbb{N}$, there exists a finite set $S\subset \mathbb{C}$ such that $|z|\in\mathbb{N}$ and $nR+C\le|z|\le (n+1)R-C$ for any $z\in S$, $|z-z'|\ge C$ for every $z,z'\in S$ with $z\ne z'$ and for each $w\in \mathbb{T}$, there exists $z\in S$ such that $|w-\frac{z}{|z|}|<\frac{\delta}{|z|}$.
\end{lemma}

\begin{lemma}[{\cite[Lemma 3.5]{Shk2}}]\label{lemS2}
Let $f\in H(\mathbb{C})$ and $l\ge 1$. There exist $m(l)\ge l$ $C_l>1$ and $\delta>0$ such that for every $a\in \mathbb{T}$, every $b\in \mathbb{R}$, every $n\in \mathbb{N}$ and every $g\in H(\mathbb{C})$, if $p_{m(l)}(f-e^{bn}T^n_{a}g)<\frac{1}{C_l}$ then 
$p_l(f-e^{cn}T^n_wg)<1$ for every $c\in\mathbb{R}$ and every $w\in \mathbb{T}$ satisfying $|b-c|<\frac{\delta}{n}$ and $|a-w|<\frac{\delta}{n}$.
\end{lemma}

\begin{prop}\label{phihyp}
There exists an increasing map $\phi:\mathbb{N}\to\mathbb{N}$ such that if $(n_k)$, $(m_k)$ and $(t_k)$ are $\phi$-increasing sequences, then for every $l,n\ge 1$, every $\varepsilon>0$, every $f\in H(\mathbb{C})$, every $K\ge 1$, there exist $x\in H(\mathbb{C})$ and $k_0,k_1,k_2\ge K$ such that $p_l(x)<\varepsilon$ and such that for every $b\in [-n,n]$, every $a\in \mathbb{T}$,
\begin{enumerate}
\item there exists $s\in [n_{k_0},n_{k_0}+\phi(n_{k_0})]$ such that
\[p_l(e^{bs}D^{s}x-f)<\varepsilon;\]
\item there exists $s\in [m_{k_1},m_{k_1}+\phi(m_{k_1})]$ such that
\[p_l(e^{bs}T^{s}_{a}x-f)<\varepsilon;\]
\item there exists $s\in [t_{k_2},t_{k_2}+\phi(t_{k_2})]$ such that
\[p_l(e^{bs}T^{s}_{1}x-f)<\varepsilon;\]
\end{enumerate}
\end{prop}
\begin{proof}

We know that the family $\{e^b D\}_{b\in [-n,n]}$ satisfies the Common Hypercyclicity Criterion for every $n$. In other words, for every $l,n\ge 1$, every $\varepsilon>0$, every $f\in H(\mathbb{C})$, every $N\ge 1$, there exist $x\in H(\mathbb{C})$ and $N_0\ge N$ such that $p_l(x)<\varepsilon$ and such that for every $b\in [-n,n]$, there exists $s\in [N,N_0]$ such that
\[p_l(e^{bs}D^{s}x-f)<\varepsilon;\]
If we consider a dense sequence $(P_n)$ in $H(\mathbb{C})$, we can thus find an increasing sequence $(\phi_n)$ such that for every $n\ge 1$, every $k\le n$, there exists $x\in H(\mathbb{C})$ such that $p_n(x)<\frac{1}{n}$ and such that for every $b\in [-n,n]$, there exists $s\in [n,n+\phi_n]$ such that
\[p_n(e^{bs}D^{s}x-P_k)<\frac{1}{n}.\]
In other words, for every $l,n\ge 1$, every $\varepsilon>0$, every $f\in H(\mathbb{C})$, there exists $M\ge 1$ such that for every $m\ge M$, there exists $x\in H(\mathbb{C})$ such that $p_l(x)<\varepsilon$ and such that for every $b\in [-n,n]$, there exists $s\in [m,m+\phi_m]$ such that
\[p_l(e^{bs}D^{s}x-f)<\varepsilon.\]
Indeed, given $l,n\ge 1$, $\varepsilon>0$ and $f\in H(\mathbb{C})$ it suffices to consider $M\ge \max\{l,n,\frac{2}{\varepsilon},n_f\}$ where $n_f$ satisfies $p_l(f-P_{n_f})<\frac{\varepsilon}{2}$.\\

We consider an increasing map $\phi:\mathbb{N}\to\mathbb{N}$ satisfying $\phi(n)\ge \phi_n$ for every $n$ and ${\sum_{m=k}^{k+\phi(k)}\frac{1}{m}\to \infty}$. Let $(n_k)$, $(m_k)$ and $(t_k)$ be $\phi$-increasing sequences and let $l,n\ge 1$, $\varepsilon>0$, $f\in H(\mathbb{C})$ and $K\ge 1$. We use Lemma~\ref{lemS1} and Lemma~\ref{lemS2} in order to determine a set of pairwise disjoint closed disks allowing us to get the desired approximations for the translation operators.

By Lemma~\ref{lemS2}, we deduce that there exist an integer $m(l)\ge l$ and real numbers $C_l>1$ and $\delta>0$ such that for every $a\in \mathbb{T}$, every $b\in \mathbb{R}$, every $m\in \mathbb{N}$ and every $g\in H(\mathbb{C})$, if $p_{m(l)}(f-e^{bm}T^m_{a}g)<\frac{1}{C_l}$ then 
$p_l(f-e^{cm}T^m_wg)<1$ for every $c\in\mathbb{R}$, $w\in \mathbb{T}$ satisfying $|b-c|<\frac{\delta}{m}$ and $|a-w|<\frac{\delta}{m}$. By using Lemma~\ref{lemS1} with $C=4m(l)$, we then obtain $R\ge 1$ and finite sets $S_m$ such that $|z|\in\mathbb{N}$ and $mR+4m(l)\le|z|\le (m+1)R-4m(l)$ for any $z\in S_m$, $|z-z'|\ge 4m(l)$ for every $z,z'\in S_m$ with $z\ne z'$, and for each $w\in \mathbb{T}$, there exists $z\in S_m$ such that $|w-\frac{z}{|z|}|<\frac{\delta}{|z|}$.

Let $i_k$ and $I_k\in \mathbb{N}$ satisfying
\[
[Ri_k,R(i_k+I_k)[\subset [m_{k},m_{k}+\phi(m_k)] \subset [R(i_k-1),R(i_k+I_k+1)[.
\]
Since $\sum_{m=k}^{k+\phi(k)}\frac{1}{m}\to \infty$, there exists $k_1\ge K$ such that
$\sum_{m=i_{k_1}}^{i_{k_1}+I_{k_1}-1}\frac{\delta R^{-1}}{m+1}> 2n$. Indeed, we have
\begin{align*}
\sum_{m=i_k}^{i_k+I_{k}-1}\frac{\delta R^{-1}}{m+1}&\ge \sum_{m=\lfloor\frac{m_k}{R}\rfloor+2}^{\lceil\frac{m_k+\phi(m_k)}{R}\rceil-1}\frac{\delta R^{-1}}{m}\\
&\ge \sum_{m=\lfloor\frac{m_k}{R}\rfloor}^{\lfloor\frac{m_k}{R}\rfloor+\lceil\frac{\phi(m_k)}{R}\rceil}\frac{\delta R^{-1}}{m}-3\frac{\delta}{m_k-1}\\
&\ge \sum_{m=m_k}^{m_k+\lceil\frac{\phi(m_k)}{R}\rceil}\frac{\delta R^{-1}}{m}-3\frac{\delta}{m_k-1}\\
&\ge \frac{1}{R}\sum_{m=m_k}^{m_k+\phi(m_k)}\frac{\delta R^{-1}}{m}-3\frac{\delta}{m_k-1}\to \infty
\end{align*}
With the same arguments, we obtain the existence of $k_2$, $j_{k_2}$ and $J_{k_2}$ with
$j_{k_2}> i_{k_1}+I_{k_1}$ such that 
\[[Rj_{k_2},R(j_{k_2}+J_{k_2})[\subset [t_{k_2},t_{k_2}+\phi(t_{k_2})]\quad\text{ and }\quad \sum_{m=j_{k_2}}^{j_{k_2}+J_{k_2}-1}\frac{\delta R^{-1}}{m+1}> 2n.\]

Since $\sum_{m=i_{k_1}}^{i_{k_1}+I_{k_1}-1}\frac{\delta R^{-1}}{m+1}> 2n$, we can now pick $b_1,\dots,b_{I_{k_1}}\in [-n,n]$ such that $[-n,n]\subset \bigcup_{k=1}^{I_{k_1}} ]b_k-\frac{\delta R^{-1}}{(i_{k_1}+k)},b_k+\frac{\delta R^{-1}}{(i_{k_1}+k)}[$
 and we can pick $c_1,\dots,c_{J_{k_2}}\in [-n,n]$ such that 
$[-n,n]\subset \bigcup_{k=1}^{J_{k_2}} ]c_k-\frac{\delta R^{-1}}{j_{k_2}+k},c_k+\frac{\delta R^{-1}}{j_{k_2}+k}[$.\\

Let $S=\bigcup_{k=1}^{I_{k_1}}S_{i_{k_1}+k-1}\cup\bigcup_{k=1}^{J_{k_2}} S_{j_{k_2}+k-1}$ and $\Lambda=S\cup\{0\}$. By definition of sets $S_m$, we have $|z-u|\ge 4m(l)$ for every $z,u\in \Lambda$. It follows that $\bigcup_{z\in \Lambda}D(z,m(l))$ is a compact set with connected complement. Let $N=\max\{|z|:z\in \Lambda\}$. We can then deduce from Mergelyan's Theorem that there exists a polynomial $P$ such that $p_{m(l)}(P)<\frac{\varepsilon}{2}$ and such that 
\[p_{m(l)}(T_zP-e^{-b_k|z|}f)<\frac{\varepsilon}{2C_l} e^{-nN}\quad\text{for each $k\in [1,I_{k_1}]$, each $z\in S_{i_{k_1}+k-1}$}\]
and
\[p_{m(l)}(T_zP-e^{-c_k|z|}f)<\frac{\varepsilon}{2C_l} e^{-nN}\quad\text{for each $k\in [1,J_{k_2}]$, each $z\in S_{j_{k_2}+k-1}$}.\]
We deduce that $p_{l}(P)<\frac{\varepsilon}{2}$, 
\[p_{m(l)}(e^{b_k|z|}T^{|z|}_{\frac{z}{|z|}}P-f)<\frac{\varepsilon}{2C_l}\quad\text{for each $k\in [1,I_{k_1}]$, each $z\in S_{i_{k_1}+k-1}$}\]
and
\[p_{m(l)}(e^{c_k|z|}T^{|z|}_{\frac{z}{|z|}}P-f)<\frac{\varepsilon}{2C_l} \quad\text{for each $k\in [1,J_{k_2}]$, each $z\in S_{j_{k_2}+k-1}$}.\]
Let $b\in [-n,n]$ and $a\in \mathbb{T}$. There exists $k\in [1,I_{k_1}]$ such that $|b-b_k|<\frac{\delta R^{-1}}{(i_{k_1}+k)}$. By definition of sets $S_m$, we can also find $z\in S_{i_{k_1}+k-1}$ such that $\big|a-\frac{z}{|z|}\big|<\frac{\delta}{|z|}$. Since $|z|<R(i_{k_1}+k)$, we have $|b-b_k|<\frac{\delta}{|z|}$ and thus by definition of $m(l)$
\[p_l(e^{b|z|}T^{|z|}_a P-f)<\frac{\varepsilon}{2}.\]
In particular, since $|z|\in[R i_k,R(i_k+I_k)[\subset [m_{k_1},m_{k_1}+\phi(m_{k_1})]$, we conclude that for every $b\in [-n,n]$, every $a\in \mathbb{T}$, there exists $s\in [m_{k_1},m_{k_1}+\phi(m_{k_1})]$ such that
\[p_l(e^{bs}T^{s}_{a}P-f)<\frac{\varepsilon}{2}.\]
Similarly, we deduce that for every $b\in [-n,n]$, every $a\in \mathbb{T}$, there exists $s\in [t_{k_2},t_{k_2}+\phi(t_{k_2})]$ such that $p_l(e^{bs}T^{s}_{a}P-f)<\frac{\varepsilon}{2}$ and we have thus in particular
\[p_l(e^{bs}T^{s}_{1}P-f)<\frac{\varepsilon}{2}.\]
Finally, we deduce from the equicontinuity that there exists $L\ge l$ and $C>0$ such that for every polynomial $Q$ if $p_L(Q)<C$ then
for every $b\in [-n,n]$, every $a\in \mathbb{T}$, there exists $s\in [m_{k_1},m_{k_1}+\phi(m_{k_1})]$ such that
\[p_l(e^{bs}T^{s}_{a}(P+Q)-f)<\varepsilon\]
and there exists $s\in [t_{k_2},t_{k_2}+\phi(t_{k_2})]$ such that
\[p_l(e^{bs}T^{s}_{1}(P+Q)-f)<\varepsilon.\]

Moreover, thanks to our choice of $\phi$, we know that there exists $M\ge 1$ such that for every $m\ge M$, there exists $x\in H(\mathbb{C})$ such that $p_l(x)<\varepsilon$, such that $p_L(x)<C$ and such that for every $b\in [-n,n]$, there exists $s\in [m,m+\phi(m)]$ satisfying
\[p_l(e^{bs}D^{s}x-f)<\varepsilon.\]
In particular, if we consider $k_0\ge K$ such that $n_{k_0}\ge \max\{M,\deg{P}\}$, there exists $Q\in H(\mathbb{C})$ such that $p_l(Q)<\varepsilon$, such that $p_L(Q)<C$ and such that for every $b\in [-n,n]$, there exists $s\in [n_{k_0},n_{k_0}+\phi(n_{k_0})]$ such that
\[p_l(e^{bs}D^{s}(P+Q)-f)<\varepsilon\quad \text{since $D^sP=0$.}\]
Since $p_L(Q)<C$, the vector $x:=P+Q$ then satisfies all desired properties.
\end{proof}

We are now able to prove the existence of common hypercyclic subspaces for the family $\{\lambda D\}_{\lambda\in\mathbb{C}\backslash\{0\}}\cup \{\mu T_a\}_{a,\mu\in\mathbb{C}\backslash\{0\}}$.

\begin{proof}[Proof of Theorem~\ref{mainthm}]
Let $\phi:\mathbb{N}\to\mathbb{N}$ be the map given by Proposition~\ref{phihyp}. We use Proposition~\ref{propX0gen} with $\phi$ in order to obtain the existence of a dense subset $X_0$ of $H(\mathbb{C})$ and three increasing sequences $(n_k)$, $(m_k)$ and $(t_k)$ such that for every $b>0$, every $\theta\in ]0,\pi[$, every $x\in X_{0}$,
\begin{enumerate}
\item $e^{bk}D^{k}x \xrightarrow[k\to \infty]{k\in \mathcal{N}} 0$;\\
\item $e^{bk}T^{k}_{a}x\xrightarrow[k\to \infty]{k\in \mathcal{M}} 0$ uniformly on $a\in \mathbb{T}\backslash \{e^{i\alpha}:\alpha\in ]-\theta,\theta[\}$;\\
\item $e^{bk}T^{k}_{1}x\xrightarrow[k\to \infty]{k\in \mathcal{T}} 0$;
\end{enumerate}
where $\mathcal{N}=\bigcup_{k\ge 1}[n_{k},n_{k}+\phi(n_{k})]$,
$\mathcal{M}=\bigcup_{k\ge 1}[m_{k},m_{k}+\phi(m_{k})]$
and $\mathcal{T}=\bigcup_{k\ge 1}[t_{k},t_{k}+\phi(t_{k})]$.

We then use Proposition~\ref{propMkgen} in order to obtain increasing sequences $(k_{1,s})$ and $(k_{2,s})$, and a non-increasing sequence of infinite-dimensional closed subspaces $(M_s)$ of $H(\mathbb{C})$ such that for every $b>0$, every $\theta\in ]0,\pi[$, every $j\ge 1$, there exists a continuous seminorm $q$ of $H(\mathbb{C})$ such that for every $a\in \mathbb{T}\backslash \{e^{i\alpha}:\alpha\in ]-\theta,\theta[\}$, every $s\in\mathbb{N}$, every $x\in M_s$, we have
\begin{enumerate}
\item for every $k\in [n_{s},n_{s}+\phi(n_{s})]$, $p_j(e^{bk}D^{k}x) \le q(x)$;
\item for every $k\in [m_{k_{1,s}},m_{k_{1,s}}+\phi(m_{k_{1,s}})]$, $p_j(e^{bk}T^{k}_ax) \le q(x)$;
\item for every $k\in [t_{k_{2,s}},t_{k_{2,s}}+\phi(t_{k_{2,s}})]$, $p_j(e^{bk}T^{k}_1x) \le q(x)$.
\end{enumerate} 

Let $(\tilde{n}_k)$ be the increasing enumeration of $\bigcup_{s\ge 1}[n_{s},n_{s}+\phi(n_{s})]$, let $(\tilde{m}_k)$ be the increasing enumeration of $\bigcup_{s\ge 1}[m_{k_{1,s}},m_{k_{1,s}}+\phi(m_{k_{1,s}})]$ and let $(\tilde{t}_k)$ be the increasing enumeration of $\bigcup_{s\ge 1}[t_{k_{2,s}},t_{k_{2,s}}+\phi(t_{k_{2,s}})]$. We consider the family $\{(T_{k,\lambda})_{k\ge 1}\}_{\lambda\in \Lambda}$ where 
\begin{itemize}
\item $\Lambda:=\mathbb{R}\cup(\T\times\mathbb{R})$, 
\item $T_{k,b}=e^b D^{\tilde{n}_k}$ for every $b\in \mathbb{R}$,
\item  $T_{k,(a,b)}=e^bT^{\tilde{m}_k}_a$ for every $a\in\T\backslash\{1\}$, every $b\in\R$ 
\item $T_{k,(1,b)}=e^bT^{\tilde{t}_k}_1$ for every $b\in\R$,
\end{itemize}
and we show that each assumption of Theorem~\ref{NiceMn} is satisfied for this family.

\begin{enumerate}
\item[(i)] Since the family $(T_a)_{a\in\mathbb{T}}$ is equicontinuous, we easily deduce that for every $k,n\ge 1$, the family
$
\{  T_{k, \lambda} \}_{\lambda\in [-n,n]\cup (\mathbb{T},[-n,n])} \ \mbox{ is equicontinuous;}
$
\
\item[(ii)]\ Let $\Lambda_n=[-n,n]\cup\big((\mathbb{T}\backslash\{e^{i\theta}:|\theta|<\frac{1}{n}\})\times [-n,n]\big)\cup\big(\{1\}\times [-n,n]\big)$. By definition of $X_0$, we can deduce that for every $n\ge 1$, for every $x\in X_{0}$,
\[
T_{k, \lambda}x \xrightarrow[k\to \infty]{} 0\quad\text{uniformly on $\Lambda_n$;}
\] 
\
\item[(iii)]\ We let  $s_k=\min\{s\in \N: \tilde{n}_k\le n_{s}+\phi(n_s),\ \tilde{m}_k\le m_{k_{1,s}}+\phi(m_{k_{1,s}})\ \text{and}\  \tilde{t}_k\le t_{k_{2,s}}+\phi(t_{k_{2,s}})\}$ and
$\tilde{M}_k:=M_{s_k}$. For each $n\ge 1$, we can then deduce that the family
$\{   (T_{k,\lambda})_{k\ge 1}\}_{\lambda\in \Lambda_n}$
is uniformly equicontinuous along $(\tilde{M}_k)$, i.e.
for every $j\ge 1$, there exists a continuous seminorm $q$ of $X$ so that for every $k$, every $x\in \tilde{M}_k$, 
\[
\mbox{sup}_{\lambda\in\Lambda_n} p_j(T_{k, \lambda}x) \le q(x).
\]
Indeed, for every $n,j\ge 1$, we know that there exists a continuous seminorm $q$ of $H(\mathbb{C})$ such that for every $a\in \mathbb{T}\backslash \{e^{i\alpha}:\alpha\in ]-\frac{1}{n},\frac{1}{n}\}$, every $s\in\mathbb{N}$, every $x\in M_s$, we have
\begin{enumerate}
\item for every $k\in [n_{s},n_{s}+\phi(n_{s})]$, $p_j(e^{nk}D^{k}x) \le q(x)$;
\item for every $k\in [m_{k_{1,s}},m_{k_{1,s}}+\phi(m_{k_{1,s}})]$, $p_j(e^{nk}T^{k}_ax) \le q(x)$;
\item for every $k\in [t_{k_{2,s}},t_{k_{2,s}}+\phi(t_{k_{2,s}})]$, $p_j(e^{nk}T^{k}_1x) \le q(x)$.
\end{enumerate} 
Therefore, for every $k\ge 1$ and every $x\in \tilde{M}_k$, if we assume that $\tilde{n}_k\in [n_{s^{(0)}},n_{s^{(0)}}+\phi(n_{s^{(0)}})]$, that $\tilde{m}_k\in [m_{k_{1,s^{(1)}}},m_{k_{1,s^{(1)}}}+\phi(m_{k_{1,s^{(1)}}})]$ and  that $\tilde{t}_k\in [t_{k_{2,s^{(2)}}},t_{k_{2,s^{(2)}}}+\phi(t_{k_{2,s^{(2)}}})]$, we have $x\in M_{s^{(0)}}\cap M_{s^{(1)}}\cap M_{s^{(1)}}$ and thus
\[
\mbox{sup}_{\lambda\in\Lambda_n} p_j(T_{k, \lambda}x) \le q(x).
\]
\end{enumerate}

We deduce from Theorem~\ref{NiceMn} that for every map $\tilde{\phi}:\mathbb{N}\to \mathbb{N}$, there exist an increasing sequence of integers $(l_i)_{i\ge 1}$ and an infinite-dimensional closed subspace $M_0$ of $H(\mathbb{C})$ such that for any $(x,\lambda)\in M_0\times \Lambda$,
\[T_{k,\lambda}x\xrightarrow[k\to \infty]{k\in I} 0,\]
where $I=\bigcup_{i\ge 1}[l_i,l_i+\tilde{\phi}(l_i)]$. 
In particular, we can choose $\tilde{\phi}$ such that for every increasing sequence of integers $(l_i)_{i\ge 1}$, we have $\{\tilde{n}_k:k\in I\}\supset \bigcup_{s}[n_{K_{0,s}},n_{K_{0,s}}+ \phi(n_{K_{0,s}})]$,
$\{\tilde{m}_k:k\in I\}\supset \bigcup_{s}[m_{K_{1,s}},m_{K_{1,s}}+\phi(m_{K_{1,s}})]$ and $\{\tilde{t}_k:k\in I\}\supset \bigcup_{s}[t_{K_{2,s}},t_{K_{2,s}}+\phi(t_{K_{2,s}})]$ for some increasing sequence $(K_{0,s})$, some subsequence $(K_{1,s})$ of $(k_{1,s})$ and some subsequence $(K_{2,s})$ of $(k_{2,s})$.

It now remains to show that the family $\{(T_{k, \lambda})_{k\in I}\}$ satisfies the assumptions of Theorem~\ref{M0gen}. We have already shown that the family $\{(T_{k, \lambda}x)_{k\ge 1}\}$ satisfies the conditions (i) and (ii). We can thus deduce that the family $\{(T_{k, \lambda}x)_{k\in I}\}$ also satisfies these conditions.  Moreover, the condition (iii) is satisfied for $\Lambda_n=\lambda\in [-n,n]\cup (\mathbb{T},[-n,n])$. This follows from our choice of $\phi$ (coming from Proposition~\ref{phihyp}) and our choice of $\tilde{\phi}$. Finally, by definition of $I$, we also know that the condition (iv) is satisfied. We can thus apply Theorem~\ref{M0gen} to the family $\{(T_{k, \lambda})_{k\in I}\}$ and get the existence of a common hypercyclic subspace for the family $\{\lambda D\}_{\lambda\in\mathbb{C}\backslash\{0\}}\cup \{\mu T_a\}_{a,\mu\in\mathbb{C}\backslash\{0\}}$.
\end{proof}

\end{document}